\documentclass[preprint,10pt]{elsarticle}

\usepackage[]{caption2}
\usepackage{graphicx}
\usepackage{subfigure}
\usepackage{amsmath}
\usepackage{amssymb}
\usepackage{marvosym}
\usepackage{booktabs}
\usepackage{threeparttable}

\usepackage{algorithm}
\usepackage{algpseudocode}
\algloopdefx{RETURN}[1][]{\textbf{return} #1}
\usepackage{color}

\newtheorem{theorem}{Theorem}[section]

\newproof{proof}{\textbf{Proof}}
\newtheorem{lemma}{Lemma}

\newtheorem{corollary}[theorem]{Corollary}

\usepackage{color}
\usepackage{url}
\usepackage[dvipdfm]{hyperref}

\journal{Discrete Applied Mathematics}

\begin{document}


\begin{frontmatter}

\title{The expected subtree number index in random polyphenylene and spiro chains}

\author[YY]{Yu~Yang}
\ead{yangyu@pdsu.edu.cn}
\author[xiaojun]{Xiao-Jun Sun}
\ead{pdschris@pdsu.edu.cn}
\author[JC]{Jia-Yi Cao}
\ead{jiayicao0219@163.com}
\author[Hwang]{Hua Wang}
\ead{hwang@georgiasouthern.edu}
\author[XD]{Xiao-Dong Zhang\corref{cor1}}
\ead{xiaodong@sjtu.edu.cn}

\cortext[cor1]{Corresponding author}
\address[YY]{School of Computer Science, Pingdingshan University, Pingdingshan 467000, China}
\address[xiaojun]{School of Electrical and Mechanical Engineering, Pingdingshan University, Pingdingshan 467000, China}
\address[JC]{Navigation College, Dalian Maritime University, Dalian, 116026, China}
\address[Hwang]{Department of Mathematical Sciences, Georgia Southern University, Statesboro, GA 30460, USA}
\address[XD]{School of Mathematical Sciences, MOE-LSC, SHL-MAC, Shanghai Jiao Tong University, Shanghai, 200240, China}




\begin{abstract}

The subtree number index $\emph{STN}(G)$ of a simple graph $G$ is the number of nonempty subtrees of $G$. It is a structural and counting topological index that has received more and more attention in recent years. In this paper we first obtain exact formulas for the expected values of subtree number index of random polyphenylene and spiro chains, which are molecular graphs of a class of unbranched multispiro molecules and polycyclic aromatic hydrocarbons. Moreover, we establish a relation between the expected values of the subtree number indices of a random polyphenylene and its corresponding hexagonal squeeze. We also present the average values for subtree number indices with respect to the set of all polyphenylene and spiro chains with $n$ hexagons.
\end{abstract}


\end{frontmatter}



\section{Introduction}
\label{Section:Terminology}

The subtree number index $\emph{STN}(G)$ of a graph $G$ is a structure-based index, defined as the total number of non-empty subtrees of $G$. It is discovered to have applications in the design of reliable communication network \cite{xiao2017trees}, bioinformatics \cite{knudsen2003optimal}, and characterizing physicochemical and structural properties of molecular graphs \cite{merrifield1989topological,yang2015subtrees,yang2017algorithms}. In recent years there have been related works on enumerating subtrees \cite{yan06,sze05,czabarka2018number,chin2018subtrees,ZHANGeccsub2019}, characterizing extremal graphs and values \cite{szekely2007binary, zhang2015minimal, kirk2008largest,zhang2013number}, analyzing relations with other topological indices such as the Wiener index  \cite{yang2015subtrees,yang2017algorithms,szekely2014extremal,wagner2007correlation}, average order and density of subtrees  \cite{vince2010average,jam1983average,haslegrave2014extremal}.

Polyphenylenes, spiro compounds and their derivatives are important polycyclic aromatic hydrocarbons in organic chemistry and have many applications in industry including organic synthesis, drug synthesis, heat exchangers, etc. For more details one may see \cite{deng2012wiener,dovslic2010chain,chen2009six,li2012hosoya} and the references cited therein.

Regarding topological indices of random polyphenylene and spiro chains, Yang and Zhang \cite{Yang2012371} found the expected value of the Wiener index of a random polyphenylene chain. Huang, Kuang and Deng \cite{Huang2013The} obtained the expected values of the Kirchhoff
index of random polyphenyl and spiro chains. Subsequently, Huang, Kuang and Deng \cite{Huang2016ms} presented explicit formulas for the
expected values of the Hosoya index and the Merrifield-Simmons index of a random polyphenylene chain. More recently, Liu \cite{Wei2018Comparing} presented explicit formulas for the expected values of ABC and GA indices in random spiro chains and compare the expected values of these two indices. Zhang, Li, Li, and Zhang \cite{ZHANG2019four} established explicit analytical expressions for the expected values of the Schultz index, Gutman index, multiplicative degree-Kirchhoff index and additive degree-Kirchhoff index of a random polyphenylene chain.


As far as the subtree number is concerned, there is no mathematical or computational studies on these two random chains. In this paper, we fill in the gap by studying the subtree number index of the random polyphenylene and spiro chains.

The rest of the paper is organized as follows. Section \ref{sec:generalnotationsandrelated} contains the necessary definitions and lemmas. In Section \ref{sec:subtreebcsubtreeofphcandspc}, we provide the expected value of the subtree number index of random polyphenylene and spiro chains, and a relation between the subtree number of these two random chains. Lastly, we briefly discuss the average value of the subtree number index of the polyphenylene and spiro chains in Section \ref{sec:averofspiroandpolyph}.

\section{Preliminaries}
\label{sec:generalnotationsandrelated}
We first introduce the technical notations and lemmas that will be used in the discussion. For more background information one may check \cite{yang2015subtrees,yan06,Wei2018Comparing,ZHANG2019four}.

Let $G =(V(G) ,E(G);f,g)$ be a weighted graph on $n$ vertices and $m$ edges, with vertex-weight function $f:V(G)\rightarrow \Re$ and edge-weight function $g:E(G)\rightarrow \Re$ (where $\Re$ is a commutative ring with a unit element 1). Denote by $\mathcal{ST}(G)$ the set of all nonempty subtrees of $G$. Given vertex subset $V_S\subseteq V(G)$ and edge subset $E_S\subseteq E(G)$, denote by $\mathcal{ST}(G,V_S)$, $\mathcal{ST}(G,E_S)$ the set of subtrees containing $V_S$, $E_S$ respectively.

For a given subtree $T\in\mathcal{ST}(G)$, its {\em weight} is defined as
\[
\omega(T)=\prod_{v\in V(T),\;e\in E(T)}f(v)g(e).
\]
And we define the {\em subtree generating function} of $G$ by
\[
F(G; f, g)=\sum_{T\in \mathcal{ST}(G)}\omega(T).
\]
Similarly, the {\em subtree generating function of $G$ containing $V_S$, $E_S$} are as follows:
\[
F(G;f, g;V_S)=\sum_{T\in \mathcal{ST}(G,V_S)}\omega(T),
\]
\[
F(G;f,g;E_S)=\sum \limits_{T\in \mathcal{ST}(G,E_S)}\omega(T).
\]

Letting $\eta(\cdot)$ be the number of subtrees in set $\mathcal{ST}(\cdot)$, we have
\[
 \eta(G) = F(G;1,1)
\]
and
\[
\eta(G,V_S)= F(G; 1,1;V_S), \quad \eta(G,E_S)= F(G; 1,1; E_S).
\]

Through introducing the subtree weight and using generating function, Yan and Yeh~\cite{yan06} developed algorithms for counting the subtrees under various constrains. We summarize their approach as follows.

Let $T =(V(T),E(T);f,g)$ be a weighted tree on $n\geq 2$ vertices, assume $u$ is a leaf vertex and $p_u =(u, v)$ is a pendant edge of $T$, we define a weighted tree $T' =(V(T'),E(T'); f',g')$ from $T$ with $V(T')=V(T)\backslash u$, $E(T')=E(T)\backslash p_u$,
\begin{align}\label{equ:genertreetrans}
f'(w) & =\begin{cases}
f(v)(1+f(u)g( p_u))&\text{if~} w=v,\\
f(w)&\text{otherwise}.
\end{cases}
\end{align}
for any $w\in V(T')$, and $g'(e)=g(e)$ for any $e\in E(T')$.

\begin{lemma}[\cite{yan06}]
\label{lemma:subtreecondes}
Assume $T$ and $T'$ are weighted trees defined above, and $u$ $(\neq v_i)$ is an arbitrary vertex, then
\[
F(T; f, g; v_i)=F(T'; f', g'; v_i),
\]
\[
F(T; f, g)=F(T'; f', g')+f(u).
\]

\end{lemma}

Assume $v_0$, $v_l$ are two distinct vertices of weighted tree $T =(V(T) ,E(T);f,g)$, denote by $P_{v_0v_l}=v_0v_1\cdots v_l$ the unique path of length $l$ $(\geq 1)$ connecting $v_0$ and $v_l$ with $V(P_{v_0v_l})=\{v_i|i=0,1,\dots,l\}$, $E(P_{v_0v_l})=\{(v_i,v_{i+1})|i=0,1,\dots,l-1\}$.  Moreover, denote by $T_{v_i}$ the weighted subtree that contains vertex ${v_i}~(i=0,1,\dots,l)$ after removing all edges in $E(P_{v_0v_l})$ from $T$.

\begin{lemma}[\cite{yan06}]\label{lemma:suboddevencontwovertexlem}
With the above notations, we have
\begin{equation}
\begin{split}
&F(T; f, g;v_0,v_l)=\prod\limits_{i=0}^{l}f^*(v_i)\prod\limits_{e\in E(P_{v_0v_l})}g(e),
\end{split}
\end{equation}
where $f^*(v_i)=F(T_{v_i}; f, g;v_i)$ for any $v_i \in V(P_{v_0v_l})$.
\end{lemma}

Let $U_{m,n}=(V(U_{m,n}), E(U_{m,n});f_{m,n},g_{m,n})$ be a weighted unicyclic graph of order $m\geq 2$ whose unique cycle has vertices $v_1,\ldots,v_n$ with $m\geq n\geq 2$. At each vertex $v_i$ let the tree $T_{v_i}$ rooted at $v_i$ be the component containing $v_i$ after removing the cycle. We can contract each tree $T_{v_i}$ to the vertex $v_i$ on the unique cycle with the contraction rule as in eq. \eqref{equ:genertreetrans} to make the computing more efficient.

For convenience, denote by $U_{n}=(V(U_n), E(U_n);f_n,g_n)$ the weighted unicyclic graph obtained from $U_{m,n}=(V(U_{m,n}), E(U_{m,n});f_{m,n},g_{m,n})$ by contracting each tree $T_{v_i}$ to the vertex $v_i$ on the unique cycle defined above, where $V(U_{n})=\{v_i|i=1,2,\dots,n\}$, $E(U_n)=\{(v_i,v_{i+1})|i=1,2,\dots,n\}$ (assume $v_{n+i}=v_{i}$ for all $i$), and $f_n$ is the vertex weight function and $g_n$ is the edge weight function, note that the weight of each vertex (resp. edge) may be different with each other.  Through classifying the subtrees of $U_n$ into $n$ types: subtrees that do not contain the edge $(v_1,v_n)$; subtrees that contain the edge set $\bigcup\limits_{k=-1}^{j}(v_{n-k},v_{n-k-1})$, but not edge $(v_{n-j-1},v_{n-j-2})$, where $v_{n+1}=v_1$ and $j=-1,0,\dots,n-3$. From the definitions of subtree weight and subtree generating function, Lemmas \ref{lemma:subtreecondes} and \ref{lemma:suboddevencontwovertexlem}, it is not difficult to obtain the following theorem.

\begin{theorem}\label{theorem:subtreegenfunofunicy}
Let $U_n=(V(U_n),E(U_n);f_n,g_n)$ be a weighted unicyclic graph, then
{\small\begin{equation}\label{equ:subtreegenfunofunicy}
\begin{split}
F(U_n; f_n, g_n)=&f_n(v_1)\Big(\prod\limits_{j=1}^{n-1}g_n(v_{n-j+2},v_{n-j+1})f_n(v_{n-j+1})\sum\limits_{s=j+1}^{n}\prod\limits_{k=1}^{n-s}g_n(v_{k},v_{k+1})f_n(v_{k+1})\Big)\\
&+\sum_{j=1}^{n-1}f_n(v_j)\big(\sum\limits_{s=1}^{n-j+1}\prod\limits_{k=j}^{n-s}g_n(v_{k},v_{k+1})f_n(v_{k+1})\big).
\end{split}
\end{equation}}
\end{theorem}

Through similar analysis, for any fixed vertex $v_j\in V(U_{n})$, the subtree generating function of $U_{n}$ containing $v_j$ follows immediately.

\begin{theorem}\label{theorem:subtrenumthreecontainfixver}
Given weighted unicyclic graph $U_n=(V(U_n),E(U_n);f_n,g_n)$ and a fixed vertex $v_j\in V(U_{n})$, we have
{\small\begin{equation}
\label{equ:subtrenumthreecontainfixver}
\begin{split}
F(U_n; f_n, g_n; v_j)=\sum\limits_{q=1}^{n}\Big(f_n(v_j)\prod\limits_{k=0}^{q-2}g_n(v_k^1,v_k^2)f_n(v_k^2)\big(1+\sum\limits_{s=0}^{n-q-1}\prod\limits_{k=0}^{s}g_n(v_k^3,v_k^4)f_n(v_k^4)\big)\Big),
\end{split}
\end{equation}}
where $v_k^1=v_{(j+k)(\text{mod } n)}$, $v_k^2=v_{(j+k+1)(\text{mod } n)}$,  $v_k^3=v_{(n+j-k)(\text{mod } n)}$ and $v_k^4=v_{(n+j-k-1)(\text{mod } n)}$.
\end{theorem}

Let $P_{v_{r_i}v_{r_j}}=v_{r_i}\cdots v_{r_j}$ be a path of $U_n$, we define the weighted unicyclic graph $U_n^c=(V(U_n^c),E(U_n^c);f_n^c,g_n^c)$ by contracting the path $P_{v_{r_i}v_{r_j}}=v_{r_i}\cdots v_{r_j}$ to $v_{r_i}$, with $V(U_n^c)=\{v_{r_i}\}\cup\{V(U_n)\backslash V(P_{v_{r_i}v_{r_j}})\}$, $E(U_n^c)=E(U_n)\backslash E(P_{v_{r_i}v_{r_j}})$,
\[
f_n^c(v_{r_i})=\prod\limits_{v\in V(P_{v_{r_i}v_{r_j}})}f_n(v)\prod\limits_{e\in E(P_{v_{r_i}v_{r_j}})}g_n(e),
\]
$f_n^c(v)=f_{n}(v)$ for $v\in V(U_n^c)\backslash v_{r_i}$, and $g_n^c(e)=g_{n}(e)$ for $e\in E(U_n^c)$. From the definitions of subtree weight and subtree generating function, with Theorem \ref{theorem:subtrenumthreecontainfixver}, we can obtain the subtree generating function of $U_{n}$ containing path $P_{v_{r_i}v_{r_j}}$ as follows.

\begin{theorem}\label{theorem:subtrenumthreecontainpath}
Assume $U_n$ and $U_n^c$ are weighted unicyclic graphs defined above, and $P_{v_{r_i}v_{r_j}}=v_{r_i}\cdots v_{r_j}$ a path of $U_n$, then
{\small\begin{equation}
\label{equ:subtrenumthreecontainpath}
\begin{split}
F(U_n; f_n, g_n; P_{v_{r_i}v_{r_j}})=F(U_n^c;f_n^c,g_n^c; v_{r_i}).
\end{split}
\end{equation}}
\end{theorem}

Assume $v_i$, $v_j$ are two distinct vertices of $U_n$, the two paths connecting $v_i$ and $v_j$ are denoted by
\[
P_{v_{i}v_{j}}^1=v_{i}v_{(i+1)(\text{mod } n)}\cdots v_{(j-1)(\text{mod } n)}v_{j}
\]
and
\[
P_{v_{i}v_{j}}^2=v_{i}v_{(i-1)(\text{mod } n)}\cdots v_{(j+1)(\text{mod } n)}v_{j},
\]
respectively.

We define the weighted unicyclic graph $U_n^{c_1}=(V(U_n^{c_1}),E(U_n^{c_1});f_n^{c_1},g_n^{c_1})$ from $U_n$ by contracting the path $P_{v_{i}v_{j}}^1$ to $v_{i}$, with $V(U_n^{c_1})=\{v_{i}\}\cup\{V(U_n)\backslash V(P_{v_{i}v_{j}}^1)\}$, $E(U_n^{c_1})=E(U_n)\backslash E(P_{v_{i}v_{j}}^1)$,
\[
f_n^{c_1}(v_{i})=\prod\limits_{v\in V(P_{v_{i}v_{j}}^1)}f_n(v)\prod\limits_{e\in E(P_{v_{i}v_{j}}^1)}g_n(e),
\]
$f_n^{c_1}(v)=f_{n}(v)$ for $v\in V(U_n^{c_1})\backslash v_{i}$, and $g_n^{c_1}(e)=g_{n}(e)$ for $e\in E(U_n^{c_1})$.

Similarly, we define the weighted unicyclic graph $U_n^{c_2}=(V(U_n^{c_2}),E(U_n^{c_2});f_n^{c_2},g_n^{c_2})$ from $U_n$ by contracting the path $P_{v_{i}v_{j}}^2$ to $v_{i}$, with $V(U_n^{c_2})=\{v_{i}\}\cup\{V(U_n)\backslash V(P_{v_{i}v_{j}}^2)\}$, $E(U_n^{c_2})=E(U_n)\backslash E(P_{v_{i}v_{j}}^2)$,
\[
f_n^{c_2}(v_{i})=\prod\limits_{v\in V(P_{v_{i}v_{j}}^2)}f_n(v)\prod\limits_{e\in E(P_{v_{i}v_{j}}^2)}g_n(e),
\]
$f_n^{c_2}(v)=f_{n}(v)$ for $v\in V(U_n^{c_2})\backslash v_{i}$, and $g_n^{c_2}(e)=g_{n}(e)$ for $e\in E(U_n^{c_2})$.

From the definitions of subtree weight and subtree generating function, with Theorem \ref{theorem:subtrenumthreecontainpath}, we can obtain the subtree generating function of $U_{n}$ containing any prescribed two distinct vertices of $U_n$ as follows.
\begin{corollary}\label{coro:subtrenumthreecontaintwovert}
Assume $U_n$, $U_n^{c_1}$ and $U_n^{c_2}$ are weighted unicyclic graph defined above, and $v_i$, $v_j$ are two prescribed distinct vertices of $U_n$, then
{\small\begin{equation}
\label{equ:subtrenumthreecontaintwovert}
\begin{split}
F(U_n; f_n, g_n; v_i,v_j)=F(U_n^{c_1};f_n^{c_1},g_n^{c_1}; v_{i})+F(U_n^{c_2};f_n^{c_2},g_n^{c_2}; v_{i}).
\end{split}
\end{equation}}
\end{corollary}

A polyphenylene chain $RPC_{n}$ with $n$ hexagons can be obtained by adjoining a polyphenylene chain $RPC_{n-1}$ with $n-1$ hexagons and a new terminal hexagon $H_n$ with a cut edge (see Fig.~\ref{fig:aspiroandpolylene}), for $n\geq3$, the terminal hexagon can be attached in three different ways, which results in the local arrangements we describe as $RPC_{n+1}^1$, $RPC_{n+1}^2$, $RPC_{n+1}^3$ (see Fig.~\ref{fig:threekindsofranrpc}).

A random polyphenylene chain $RPC(n,p_1,p_2)$ with $n$ hexagons is a polyphenylene chain obtained by step-wise addition of terminal hexagons. At each step $(i=3, 4,\dots, n)$, a random selection is made from one of the three possible constructions:

(1) $RPC_{i-1}\rightarrow RPC_{i}^1$ with probability $p_1$,

(2) $RPC_{i-1}\rightarrow RPC_{i}^2$ with probability $p_2$,

(3) $RPC_{i-1}\rightarrow RPC_{i}^3$ with probability $1-p_1-p_2$,

Here the probabilities $p_1$ and $p_2$ are constants. Namely, the process described is a zeroth-order Markov Process.

\begin{figure}[!htbp]
  \centering
  \subfigure[A polyphenylene chain $RPC_n$ with $n$ hexagons.]{
    \label{fig:polylenechain:a}
    \includegraphics[width=0.35\textwidth]{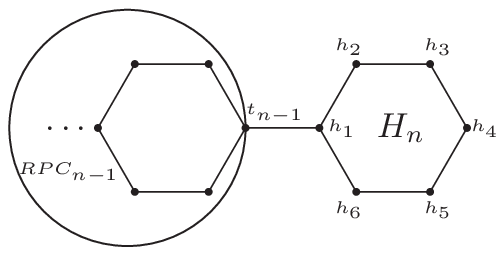}}
  \hspace{0.8in}
  \subfigure[A spiro chain $RSC_n$ with $n$ hexagons.]{
    \label{fig:spriochain:b}
    \includegraphics[width=0.4\textwidth]{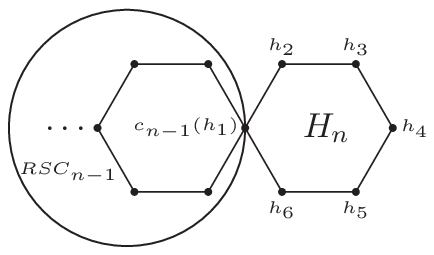}}
  \caption{A polyphenylene chain $RPC_n$ and a spiro chain $RSC_n$.}
  \label{fig:aspiroandpolylene} 
\end{figure}

\begin{figure}[!htbp]
\centering
\includegraphics[width=\textwidth]{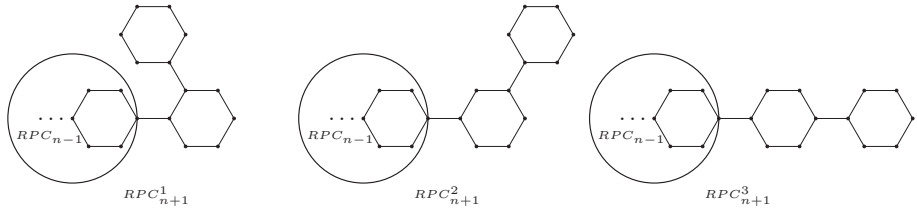}\\
\caption{Three types of local arrangements in polyphenylene chains.}
\label{fig:threekindsofranrpc}
\end{figure}

\begin{figure}[!htbp]
\centering
\includegraphics[width=\textwidth]{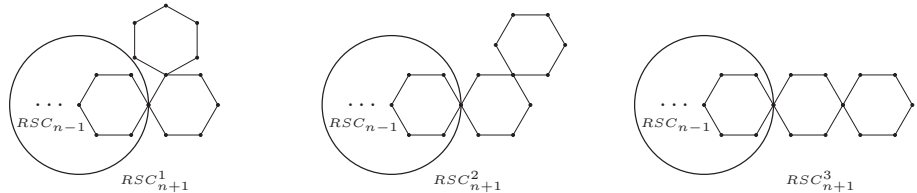}\\
\caption{Three types of local arrangements in spiro chains.}
\label{fig:threekindsofranrsc}
\end{figure}

Similarly, a spiro chain $RSC_{n}$ with $n$ hexagons can be obtained by adjoining a new terminal hexagon $H_n$ to a spiro chain $RSC_{n-1}$ with $n-1$ hexagons (see Fig.~\ref{fig:aspiroandpolylene}), for $n\geq3$, the terminal hexagon can be attached in three different ways, which results in the local arrangements we describe as $RSC_{n+1}^1$, $RSC_{n+1}^2$, $RSC_{n+1}^3$ (see Fig.~\ref{fig:threekindsofranrsc}).

And a random spiro chain $RSC(n,p_1,p_2)$ with $n$ hexagons is a spiro chain obtained by stepwise addition of terminal hexagons. At each step $(i=3, 4,\dots, n)$, a random selection is made from one of the three possible constructions:

(1) $RSC_{i-1}\rightarrow RSC_{i}^1$ with probability $p_1$,

(2) $RSC_{i-1}\rightarrow RSC_{i}^2$ with probability $p_2$,

(3) $RSC_{i-1}\rightarrow RSC_{i}^3$ with probability $1-p_1-p_2$,

Here the probabilities $p_1$ and $p_2$ are constants.

Specially, the random polyphenylene chain $RPC(n,1,0)$,  $RPC(n,0,1)$ and $RPC(n,0,0)$ are the polyphenylene ortho-chain $\overline{O}_n$, meta-chain $\overline{M}_n$, para-chain $\overline{P}_n$, by setting $(p_1,p_2)=(1,0),(0,1)$, $(0,0)$, respectively, see Fig.~\ref{fig:threekindsofranrpc}. And similarly, the random spiro chain $RSC(n,1,0)$, $RSC(n,0,1)$,  $RSC(n,0,0)$ are the spiro ortho-chain $O_n$, meta-chain $M_n$ and para-chain $P_n$, respectively, see Fig.~\ref{fig:threekindsofranrsc}.

\section{The expected value of the subtree number index of random polyphenylene and spiro chains}
\label{sec:subtreebcsubtreeofphcandspc}

\subsection{Random polyphenylene chain}

Firstly, we study the subtree number index of the random polyphenylene chain.

\begin{theorem}\label{theo:expsubtreeofrpc}
For $n\geq 1$, the expected value of the subtree number index of random polyphenylene $RPC(n,p_1,p_2)$ is
\begin{equation}
\label{equ:subtreenumofrpc}
\begin{split}
E(\emph{STN}(RPC(n,p_1,p_2)))=&\frac{441}{(11+4p_1+p_2)^2}(12+4p_1+p_2)^n+\frac{144p_1+36p_2-45}{11+4p_1+p_2}n\\
&-\frac{441}{(11+4p_1+p_2)^2}.
\end{split}
\nonumber
\end{equation}

\end{theorem}
\begin{proof}
It is easy to know that $E(\emph{STN}(RPC(1,p_1,p_2)))=36$, for $n\geq 2$, we categorize the subtrees of the random polyphenylene chain $RPC_n$ into two cases:

(i) not containing edge $(t_{n-1},h_{1})$,

(ii) containing edge $(t_{n-1},h_{1})$.

It is easy to see that the subtree number of case (i) is
\begin{equation}
\label{equ:subtreenumofrpc01}
\begin{split}
\emph{STN}(RPC_{n-1})+36.
\end{split}
\end{equation}

The subtree set of case (ii), denoted by $RPC_{n}(t_{n-1},h_{1})$  can be described as:
\begin{equation}
\label{equ:subtreenumofrpc02}
\begin{split}
RPC_{n}(t_{n-1},h_{1})=\{T_l+(t_{n-1},h_{1})+T_r|T_l\in RPC_{n-1}(t_{n-1}), T_r\in H_{n}(h_{1})\},
\end{split}
\end{equation}
where $T_l+(t_{n-1},h_{1})+T_r$ is the subtree obtained by connecting subtree $T_l\in RPC_{n-1}(t_{n-1})$ (subtree of $RPC_{n-1}$ containing vertex $t_{n-1}$) and subtree $T_r\in H_{n}(h_{1})$ (subtree of $H_{n}$ containing vertex $h_{1}$) with edge $(t_{n-1},h_{1})$, see Fig.~\ref{fig:polylenechain:a}.

Since the subtree number of $H_{n}$ containing vertex $h_{1}$ is 21, by eq.~\eqref{equ:subtreenumofrpc02}, we have the subtree number of case (ii) as
\begin{equation}
\label{equ:subtreenumofrpcconcomm}
\begin{split}
21\emph{STN}(RPC_{n-1}(t_{n-1})).
\end{split}
\end{equation}

Thus by eqs.~\eqref{equ:subtreenumofrpc01} and \eqref{equ:subtreenumofrpcconcomm}, we have
\begin{equation}
\label{equ:subtreenumofrpc1}
\begin{split}
\emph{STN}(RPC_n)=\emph{STN}(RPC_{n-1})+21\emph{STN}(RPC_{n-1}(t_{n-1}))+36.
\end{split}
\end{equation}

Namely,
\begin{equation}
\label{equ:subtreenumofrpc11}
\begin{split}
\emph{STN}(RPC_{n+1})=\emph{STN}(RPC_{n})+21\emph{STN}(RPC_{n}(t_{n}))+36.
\end{split}
\end{equation}

For a random polyphenylene chain $RPC(n,p_1,p_2)$, the subtree number of $RPC(n,p_1,p_2)$ containing vertex $t_{n}$ is a random variable, and its expected value is denoted by
\begin{equation}
\label{equ:subtreenumofrpc2}
\begin{split}
T_{n}=E(\emph{STN}(RPC(n,p_1,p_2;t_{n}))).
\end{split}
\end{equation}
By the expectation operator and eqs.~\eqref{equ:subtreenumofrpc11} and \eqref{equ:subtreenumofrpc2}, we can obtain a recursive relation for the expected value of the subtree number index of a random polyphenylene chain $RPC(n,p_1,p_2)$
\begin{equation}
\label{equ:subtreenumofrpc3}
\begin{split}
E(\emph{STN}(RPC(n+1,p_1,p_2)))=E(\emph{STN}(RPC(n,p_1,p_2)))+21T_{n}+36.
\end{split}
\end{equation}

Now, we consider computing $T_{n}$. Take the weighted random polyphenylene chain $RPC(n-1,p_1,p_2)$ as a single vertex $``t_{n-1}"$ whose weight is $\emph{STN}(RPC(n-1,p_1,p_2;t_{n-1}))$, see Fig.~\ref{fig:polylenechain:a}, then, with Theorem \ref{theorem:subtrenumthreecontainfixver}, we have
\begin{enumerate}[(i)]
  \item If $RPC_{n}\rightarrow RPC_{n+1}^1$ with probability $p_1$,
   $$\emph{STN}(RPC(n,p_1,p_2;t_{n}))=21+16\emph{STN}(RPC(n-1,p_1,p_2;t_{n-1}))$$
   with probability $p_1$.
  \item If $RPC_{n}\rightarrow RPC_{n+1}^2$ with probability $p_2$,
   $$\emph{STN}(RPC(n,p_1,p_2;t_{n}))=21+13\emph{STN}(RPC(n-1,p_1,p_2;t_{n-1}))$$
   with probability $p_2$.
  \item If $RPC_{n}\rightarrow RPC_{n+1}^3$ with probability $1-p_1-p_2$,
   $$\emph{STN}(RPC(n,p_1,p_2;t_{n}))=21+12\emph{STN}(RPC(n-1,p_1,p_2;t_{n-1}))$$
   with probability $1-p_1-p_2$.
\end{enumerate}

From (i)-(iii) above, we immediately obtain
\begin{equation}
\label{equ:subtreenumofrpctn-1}
\begin{split}
T_{n}=&p_1[21+16\emph{STN}(RPC(n-1,p_1,p_2;t_{n-1}))]+p_2[21+13\emph{STN}(RPC(n-1,p_1,p_2;t_{n-1}))]\\
        &+(1-p_1-p_2)[21+12\emph{STN}(RPC(n-1,p_1,p_2;t_{n-1}))]\\
        =&(12+4p_1+p_2)\emph{STN}(RPC(n-1,p_1,p_2;t_{n-1})+21.
\end{split}
\end{equation}

By applying the expectation operator to the above eq.~\eqref{equ:subtreenumofrpctn-1}, we obtain
\begin{equation}
\label{equ:subtreenumofrpctn-1step1}
\begin{split}
T_{n}=(12+4p_1+p_2)T_{n-1}+21.
\end{split}
\end{equation}

Since $T_{1}=21$, using the above recurrence relation, we have
\begin{equation}
\label{equ:subtreenumofrpctn-1step2}
\begin{split}
T_{n}=\frac{21}{11+4p_1+p_2}(12+4p_1+p_2)^{n}-\frac{21}{11+4p_1+p_2}.
\end{split}
\end{equation}

From eq.~\eqref{equ:subtreenumofrpc3}, we have
\begin{equation}
\label{equ:subtreenumofrpc3add}
\begin{split}
E(\emph{STN}(RPC(n+1,p_1,p_2)))=&21[\frac{21}{11+4p_1+p_2}(12+4p_1+p_2)^{n}-\frac{21}{11+4p_1+p_2}]\\
 &+E(\emph{STN}(RPC(n,p_1,p_2)))+36\\
 =&\frac{441}{11+4p_1+p_2}(12+4p_1+p_2)^n-\frac{441}{11+4p_1+p_2}\\
 &+E(\emph{STN}(RPC(n,p_1,p_2)))+36.
\end{split}
\end{equation}

Using the above recurrence relation, we have
\begin{equation}
\label{equ:subtreenumofrpctn-1step3}
\begin{split}
E(\emph{STN}(RPC(n,p_1,p_2)))=&\frac{441(12+4p_1+p_2)}{(11+4p_1+p_2)^2}\big((12+4p_1+p_2)^{n-1}-1\big)\\
&+(36-\frac{441}{11+4p_1+p_2})(n-1)+36.
\end{split}
\end{equation}

The theorem thus follows.

\end{proof}
Specially, by taking $(p_1,p_2)=(1,0),(0,1)$ or $(0,0)$, respectively, and Theorem \ref{theo:expsubtreeofrpc},
we have the following.

\begin{corollary}\label{cor:subtrenumthreespeciapccexac}
The subtree number indices of the polyphenylene ortho-chain $\overline{O}_n$, meta-chain $\overline{M}_n$ and para-chain $\overline{P}_n$ are
\begin{equation}
\label{equ:subtrenumberthreespeciaphc}
\begin{split}
&{STN}(\overline{O}_n)={\frac{49(16^{n}-1)}{25}}+{\frac{33n}{5}},\\
&{STN}(\overline{M}_n)={\frac{49(13^{n}-1)}{16}}-{\frac{3n}{4}},\\
&{STN}(\overline{P}_n)={\frac{441(12^{n}-1)}{121}}-{\frac{45n}{11}}.
\end{split}
\nonumber
\end{equation}
\end{corollary}

The results of Corollary \ref{cor:subtrenumthreespeciapccexac} agree with the subtree numbers of $\overline{O}_n$, $\overline{M}_n$ and $\overline{P}_n$ presented in \cite{yang2015subtrees}.

\subsection{Random spiro chain}
\label{sec:subtreebcsubnumofshcs}

\begin{theorem}\label{theo:exvaluesubtreeofrsc}
For $n\geq 1$, the expected value of the subtree number index of a random spiro chain $RSC(n,p_1,p_2)$ is
\begin{equation}
\label{equ:subtreenumofrpc}
\begin{split}
E(\emph{STN}(RSC(n,p_1,p_2)))=&\frac{400}{(11+4p_1+p_2)^2}(12+4p_1+p_2)^{n}+\frac{140p_1+35p_2-15}{11+4p_1+p_2}n\\
&-\frac{400}{(11+4p_1+p_2)^2}+1.
\end{split}
\nonumber
\end{equation}

\end{theorem}
\begin{proof}
It is not difficult to obtain that $E(\emph{STN}(RSC(1,p_1,p_2)))=36$, for $n\geq 2$, assume that the terminal hexagon is spanned by vertices $h_2$, $h_3$, $h_4$, $h_5$,
$h_6$, and the vertex $h_1$(i.e. $c_{n-1}$) (see Fig.~\ref{fig:spriochain:b}). We categorize the subtrees of the random spiro chain $RSC_n$ into four cases:

(i) contain neither $(c_{n-1},h_2)$ nor $(c_{n-1},h_6)$,

(ii) contain $(c_{n-1},h_2)$ but not $(c_{n-1},h_6)$,

(iii) contain $(c_{n-1},h_6)$ but not $(c_{n-1},h_2)$,

(iv) contain both $(c_{n-1},h_2)$ and $(c_{n-1},h_6)$.

by Lemma \ref{lemma:subtreecondes}, the subtree number of case (i) is

\begin{equation}
\label{equ:subtreenumofrsc01}
\begin{split}
\emph{STN}(RSC_{n-1})+15.
\end{split}
\end{equation}

Taking the the random spiro chain $RSC_{n-1}$ as a single vertex ``$c_{n-1}$" with weight $\emph{STN}(RSC_{n-1}(c_{n-1}))$ (namely, subtree number of $RSC_{n-1}$ containing vertex $c_{n-1}$) (see Fig.~\ref{fig:spriochain:b}), then, by Lemma \ref{lemma:suboddevencontwovertexlem}, we know that both the subtree number of case (ii) and case (iii) are
\begin{equation}
\label{equ:subtreenumofrsc023}
\begin{split}
5\emph{STN}(RSC_{n-1}(c_{n-1})).
\end{split}
\end{equation}
and further with Theorem \ref{theorem:subtrenumthreecontainpath}, we can obtain that the subtree number of case (iii) is
\begin{equation}
\label{equ:subtreenumofrsc024}
\begin{split}
10\emph{STN}(RSC_{n-1}(c_{n-1})).
\end{split}
\end{equation}
thus, we have
\begin{equation}
\label{equ:subtreenumofspc025}
\begin{split}
\emph{STN}(RSC_n)=\emph{STN}(RSC_{n-1})+20\emph{STN}(RSC_{n-1}(c_{n-1}))+15.
\end{split}
\end{equation}

Namely,
\begin{equation}
\label{equ:subtreenumofrsc026}
\begin{split}
\emph{STN}(RSC_{n+1})=\emph{STN}(RSC_{n})+20\emph{STN}(RSC_{n}(c_{n}))+15.
\end{split}
\end{equation}
For a random spiro chain $RSC(n,p_1,p_2)$, the subtree number of $RSC(n,p_1,p_2)$ containing vertex $c_{n}$ is a random variable, and its expected value is denoted by
\begin{equation}
\label{equ:subtreenumofrsc027}
\begin{split}
C_{n}=E(\emph{STN}(RSC(n,p_1,p_2;c_{n}))).
\end{split}
\end{equation}
By the expectation operator and eqs.~\eqref{equ:subtreenumofrsc026} and \eqref{equ:subtreenumofrsc027}, we can obtain a recursive relation for the expected value of the subtree number index of a random spiro chain $RSC(n,p_1,p_2)$
\begin{equation}
\label{equ:subtreenumofrsc028}
\begin{split}
E(\emph{STN}(RSC(n+1,p_1,p_2)))=E(\emph{STN}(RSC(n,p_1,p_2)))+20C_{n}+15.
\end{split}
\end{equation}

Now, we consider $C_{n}$. Again, take the the random spiro chain $RSC_{n-1}$ as a single vertex ``$c_{n-1}$" with weight $\emph{STN}(RSC_{n-1}(c_{n-1}))$ (see Fig.~\ref{fig:spriochain:b}), then, with Theorem \ref{theorem:subtrenumthreecontainfixver},
\begin{enumerate}[(i)]
  \item If $RSC_{n}\rightarrow RSC_{n+1}^1$ with probability $p_1$,
   $$\emph{STN}(RSC(n,p_1,p_2;c_{n}))=5+16\emph{STN}(RSC(n-1,p_1,p_2;c_{n-1}))$$
   with probability $p_1$.
  \item If $RSC_{n}\rightarrow RSC_{n+1}^2$ with probability $p_2$,
   $$\emph{STN}(RSC(n,p_1,p_2;c_{n}))=8+13\emph{STN}(RSC(n-1,p_1,p_2;c_{n-1}))$$
   with probability $p_2$.
  \item If $RSC_{n}\rightarrow RSC_{n+1}^3$ with probability $1-p_1-p_2$,
   $$\emph{STN}(RSC(n,p_1,p_2;c_{n}))=9+12\emph{STN}(RSC(n-1,p_1,p_2;c_{n-1}))$$
   with probability $1-p_1-p_2$.
\end{enumerate}

From (i)-(iii) above, we immediately obtain
\begin{equation}
\label{equ:subtreenumofrsctn-1}
\begin{split}
C_{n}=&p_1[5+16\emph{STN}(RSC(n-1,p_1,p_2;c_{n-1}))]+p_2[8+13\emph{STN}(RSC(n-1,p_1,p_2;c_{n-1}))]\\
        &+(1-p_1-p_2)[9+12\emph{STN}(RSC(n-1,p_1,p_2;c_{n-1}))]\\
        =&(12+4p_1+p_2)\emph{STN}(RSC(n-1,p_1,p_2;c_{n-1})+9-(4p_1+p_2).
\end{split}
\end{equation}
By applying the expectation operator to the above eq.~\eqref{equ:subtreenumofrsctn-1}, we obtain
\begin{equation}
\label{equ:subtreenumofrsctn-1step1}
\begin{split}
C_{n}=(12+4p_1+p_2)C_{n-1}+9-(4p_1+p_2).
\end{split}
\end{equation}
Since $C_{1}=21$, using the above recurrence relation, we have
\begin{equation}
\label{equ:subtreenumofrsctn-1step2}
\begin{split}
C_{n}=\frac{20}{11+4p_1+p_2}(12+4p_1+p_2)^{n}+1-\frac{20}{11+4p_1+p_2}.
\end{split}
\end{equation}

It is easy to see that $E(\emph{STN}(RSC(1,p_1,p_2)))=36$, from eq.~\eqref{equ:subtreenumofrsc028}, we have
\begin{equation}
\label{equ:subtreenumofrsc3add}
\begin{split}
E(\emph{STN}(RSC(n+1,p_1,p_2)))=&20[\frac{20}{11+4p_1+p_2}(12+4p_1+p_2)^{n}+1-\frac{20}{11+4p_1+p_2}]\\
 &+E(\emph{STN}(RSC(n,p_1,p_2)))+15\\
 =&\frac{400}{11+4p_1+p_2}(12+4p_1+p_2)^n-\frac{400}{11+4p_1+p_2}\\
 &+E(\emph{STN}(RSC(n,p_1,p_2)))+35.
\end{split}
\end{equation}

Using the above recurrence relation, we have
\begin{equation}\label{equ:subtreenumofrscfav}
\begin{split}
E(\emph{STN}(RSC(n,p_1,p_2)))=&\frac{400(12+4p_1+p_2)}{(11+4p_1+p_2)^2}\big((12+4p_1+p_2)^{n-1}-1\big)\\
&+(35-\frac{400}{11+4p_1+p_2})(n-1)+36.
\end{split}
\end{equation}
The theorem holds immediately.
\end{proof}

Let $(p_1,p_2)=(1,0),(0,1)$, or $(0,0)$, we can obtain the subtree number index of the spiro ortho-chain $O_n$, the meta-chain $M_n$ and the para-chain $P_n$, respectively, with Theorem~\ref{theo:exvaluesubtreeofrsc}, we have
\begin{corollary}\label{cor:subtrenumthreespeciashcexac}
The subtree number indices of the spiro ortho-chain $O_n$, the meta-chain $M_n$ and the para-chain $P_n$ are
\begin{equation}
\label{equ:subtrenumthreespeciashcexac}
\begin{split}
&{STN}(O_n)={\frac{256(16^{n-1}-1)}{9}}+{\frac{25(n-1)}{3}}+36,\\
&{STN}(M_n)={\frac{325(13^{n-1}-1)}{9}}+{\frac{5(n-1)}{3}}+36,\\
&{STN}(P_n)={\frac{4800(12^{n-1}-1)}{121}}-{\frac{15(n-1)}{11}}+36.\\
\end{split}
\end{equation}
\end{corollary}

Again, the results of Corollary \ref{cor:subtrenumthreespeciashcexac} agree with the subtree numbers of $O_n$, $M_n$ and $P_n$ presented in \cite{yang2015subtrees}.

\subsection{A relation between $E(\emph{STN}(RPC))$ and $E(\emph{STN}(RSC))$}
\label{Sec:relationbesubandbcsubofbothchains}

It is easy to see that every spiro chain could be obtained by  squeezing off the cut edges of a polyphenylene chain.  Pavlovi{\'{c}} and
Gutman \cite{pavlovic1997wiener}, Deng \cite{deng2012wiener} provided a formula of the relation between the Wiener indices of a polyphenylene chain and its squeeze independently. In 2015, Yang et al. \cite{yang2015subtrees} presented a formula of the relation between the subtree number index of these two chains.

When random structures are concerned, Yang and Zhang \cite{Yang2012371} presented an  exact formula for
the expected value of the Wiener index of a random polyphenylene chain $RPC(n,p_1,p_2)$ with the same probabilities $p_1$ and $p_2$, Regarding the random polyphenylene chain $RPC(n,p_1,p_2)$ and spiro chain $RSC(n,p_1,p_2)$,  Huang, Kuang and Deng \cite{Huang2013The} presented a relation between the expected values of the Kirchhoff indices of these two chains.

In what follows, we present a relation
between the expected values of the subtree number indices of the random polyphenylene chain $RPC(n,p_1,p_2)$ and the random spiro chain $RSC(n,p_1,p_2)$  from Theorems \ref{theo:expsubtreeofrpc} and \ref{theo:exvaluesubtreeofrsc}.

\begin{theorem}\label{theo:relationbetwsEXhcandphc}
 For a random polyphenylene chain $RPC(n,p_1,p_2)$ and a random spiro chain
$RSC(n,p_1,p_2)$ with $n$ hexagons, the expected values of their subtree number indices are related
as
\begin{equation}\label{equ:subtreegenfunrelationsandp}
\begin{split}
400E(\emph{STN}(RPC(n,p_1,p_2)))=441E(\emph{STN}(RSC(n,p_1,p_2)))-1035n-441.
\end{split}
\end{equation}
\end{theorem}
\begin{proof}

From Eq.~\eqref{equ:subtreenumofrpctn-1step3} and \eqref{equ:subtreenumofrscfav}, we have
\begin{equation}\label{equ:subtreegenfuncdiffence}
\begin{split}
\frac{E(\emph{STN}(RPC(n,p_1,p_2)))-(36-\frac{441}{11+4p_1+p_2})(n-1)-36}{E(\emph{STN}(RSC(n,p_1,p_2)))-(35-\frac{400}{11+4p_1+p_2})(n-1)-36}=\frac{441}{400},\\
\end{split}
\end{equation}
or equivalently,
\begin{equation}\label{equ:subtreediffofsandp}
\begin{split}
400E(\emph{STN}(RPC(n,p_1,p_2)))=441E(\emph{STN}(RSC(n,p_1,p_2)))-1035n-441.\\
\end{split}
\nonumber
\end{equation}
\qed
\end{proof}

 By  Theorem~\ref{theo:relationbetwsEXhcandphc},   the expected value of the subtree number index of the random spiro chain is less than the random polyphenylene chain. In fact,  for $n\rightarrow \infty$,
 $$E(\emph{STN}(RSC(n,p_1,p_2))) \approx \frac{400}{441}E(\emph{STN}(RPC(n,p_1,p_2))).$$

From Theorems \ref{theo:expsubtreeofrpc} and \ref{theo:exvaluesubtreeofrsc}, we also point out that
\[
E(\emph{STN}(RPC(n,p_1,p_2)))\approx \frac{441}{(11+4p_1+p_2)^2}(12+4p_1+p_2)^n
\]
and
\[
E(\emph{STN}(RSC(n,p_1,p_2)))\approx \frac{400}{(11+4p_1+p_2)^2}(12+4p_1+p_2)^{n}.
\]

Namely, the values of $E(\emph{STN}(RPC(n,p_1,p_2)))$ and $E(\emph{STN}(RSC(n,p_1,p_2)))$ are asymptotic to exponential function in $n$ as $n\longrightarrow  \infty$.

\section{Average value of the subtree number index}
\label{sec:averofspiroandpolyph}
Let $\overline{G}_n$ be the set of all polyphenylene chains with $n$ hexagons. The average value of the
subtree number indices with respect to $\overline{G}_n$ is
\[
\emph{STN}_{avr}(\overline{G}_n)=\frac{1}{|\overline{G}_n|}\sum\limits_{G\in \overline{G}_n}\emph{STN}(G).
\]

In order to obtain the average value of the subtree number indices with respect to $\overline{G}_n$, we
only need to take $p_1=p_2=\frac{1}{3}$ in the random polyphenylene chain $RPC(n,p_1,p_2)$, i.e., the
average value of the subtree number indices with respect to $\overline{G}_n$ is just the expected value of the
subtree number index of the random polyphenylene chain $RPC(n,p_1,p_2)$ for $p_1=p_2=\frac{1}{3}$. From
Theorem \ref{theo:expsubtreeofrpc}, we have
\begin{theorem}\label{theo:aversubtreeofrpc}
The average value of the subtree number indices with respect to $\overline{G}_n$ is
\begin{equation}
\label{equ:subtreenumofrpcaver}
\begin{split}
\emph{STN}_{avr}(\overline{G}_n)=\frac{3969}{1444}(\frac{41}{3})^n+\frac{45}{38}n-\frac{3969}{1444}.
\end{split}
\nonumber
\end{equation}

\end{theorem}

Similarly, Let $G_n$ be the set of all spiro chains with $n$ hexagons. The average value of the
subtree number indices with respect to $G_n$ is
\[
\emph{STN}_{avr}(G_n)=\frac{1}{|G_n|}\sum\limits_{G\in G_n}\emph{STN}(G),
\]
and the average value of the subtree number indices with respect to $G_n$ is just the expected value of the
subtree number index of the random spiro chain $RSC(n,p_1,p_2)$ for $p_1=p_2=\frac{1}{3}$. From
Theorem \ref{theo:exvaluesubtreeofrsc}, we have
\begin{theorem}\label{theo:aversubtreeofrsc}
The average value of the subtree number indices with respect to $\overline{G}_n$ is
\begin{equation}
\label{equ:subtreenumofrpcaver}
\begin{split}
\emph{STN}_{avr}(G_n)=\frac{900}{361}(\frac{41}{3})^n+\frac{130}{38}n-\frac{539}{361}.
\end{split}
\nonumber
\end{equation}

\end{theorem}

\section{Concluding remarks}
\label{Sec:Conclusion}
In this paper we obtain exact formulas for the expected values of subtree number index of the random polyphenylene and spiro chains, and then establish a relation between the expected values of the subtree number indices of a random polyphenylene and its corresponding random hexagonal squeeze, we also briefly study the average values for subtree number indices with respect to the set of all polyphenylene and spiro chains with $n$ hexagons.

For future works, we plan to study the expected values of subtree number index of the random hexagonal chains, phenylene
chains and other regular chemical structures such as cata-condensed hexagonal systems. Meanwhile, It is also interesting to study the expected values of the recently proposed multi-distance granularity structural $\alpha$-subtree index \cite{yangyu2019aerfa} of the random polyphenylene and spiro chains, as well as other regular chemical structures.

\section*{Acknowledgment}

 The authors would like to thank two anonymous referees for their suggestions and comments which results in a great improvement of this original paper.
This work is partially supported by
the National Natural Science Foundation of China (Grant nos. 61702291, 61772102, 11971311, 11531001, 11801371); Program for Science \& Technology Innovation Talents in Universities of Henan Province(Grant no. 19HASTIT029); the Montenegrin-Chinese Science and Technology Cooperation Project (Grant No. 3-12); the Key Research Project in Universities  of Henan Province(Grant nos. 19B110011, 19B630015); and the Scientific Research Starting Foundation for High-level Talents of Pingdingshan University (Grant no.PXY-BSQD2017006).

\end{document}